\documentclass[a4paper, 12pt, reqno]{amsart}
\usepackage{amsmath,amsfonts,amssymb,graphics,mathrsfs,amsthm,eurosym,verbatim,enumerate,subcaption}
\usepackage[marginratio=1:1,tmargin=117pt, height=650pt]{geometry}
\usepackage[usenames, dvipsnames]{xcolor}
\usepackage[shortlabels]{enumitem}

\usepackage[normalem]{ulem}
\usepackage{bm}
\usepackage{tikz-cd}
\usepackage[colorlinks=true,linkcolor=blue!60!black,citecolor=teal!80!black,urlcolor=RoyalBlue,hypertexnames=false]{hyperref}

\numberwithin{equation}{section}
\usepackage{cleveref}
\makeatletter
\def\blfootnote{\xdef\@thefnmark{}\@footnotetext}
\makeatother
\theoremstyle{plain}
\newtheorem{theorem}{Theorem}[section]

\newtheorem{corollary}[theorem]{Corollary}
\newtheorem{lemma}[theorem]{Lemma}
\newtheorem{definition}[theorem]{Definition}
\newtheorem{conjecture}[theorem]{Conjecture}

\newcommand*{\defeq}{\mathrel{\vcenter{\baselineskip0.5ex \lineskiplimit0pt
 \hbox{\scriptsize.}\hbox{\scriptsize.}}}=}

\newtheorem*{remark}{Remark}
\newtheorem*{thm}{Theorem A}
\newtheorem{observation}[theorem]{Observation}

\theoremstyle{remark}
\newtheorem*{claim}{Claim}
\newcommand{\C}{{\mathbb{C}}}

\newcommand{\R}{{\mathbb{R}}}
\newcommand{\Z}{{\mathbb{Z}}}
\newcommand{\N}{{\mathbb{N}}}

\newcommand{\M}{{\mathcal{M}}}

\newcommand{\qfor}{\quad\text{for }}

\newenvironment{subproof}{\begin{proof}[Proof of claim]}{%
               \end{proof}}
\newcommand{\hp}{\hat{p}}

\newcommand{\tq}{\tilde{q}}
\newcommand{\hz}{\hat{z}}
\newcommand{\id}{{\mu_p}}
\newcommand{\hid}{{\mu_{\hp}}}
\newcommand{\tk}{t_{\operatorname{kmax}}}

\begin{document}
\title[On a result of Hayman]{On a result of Hayman concerning \\ the  maximum modulus set}
\author[{V. Evdoridou \and L. Pardo-Sim\'on \and D. J. Sixsmith}]{Vasiliki Evdoridou \and Leticia Pardo-Sim\'on \and David J. Sixsmith}
\address{School of Mathematics and Statistics\\ The Open University\\
Milton Keynes MK7 6AA\\ UK\textsc{\newline \indent \href{https://orcid.org/0000-0002-5409-2663}{\includegraphics[width=1em,height=1em]{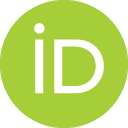} {\normalfont https://orcid.org/0000-0002-5409-2663}}}}
\email{vasiliki.evdoridou@open.ac.uk}
\address{Institute of Mathematics of the Polish Academy of Sciences\\ ul. \'Sniadeckich~8\\
00-656 Warsaw\\ Poland\textsc{\newline \indent \href{https://orcid.org/0000-0003-4039-5556}{\includegraphics[width=1em,height=1em]{orcid2.png} {\normalfont https://orcid.org/0000-0003-4039-5556}}}
}
\email{l.pardo-simon@impan.pl}
\address{School of Mathematics and Statistics\\ The Open University\\
Milton Keynes MK7 6AA\\ UK\textsc{\newline \indent \href{https://orcid.org/0000-0002-3543-6969}{\includegraphics[width=1em,height=1em]{orcid2.png} {\normalfont https://orcid.org/0000-0002-3543-6969}}}}
\email{david.sixsmith@open.ac.uk}
\dedicatory{This paper is dedicated to the memory of Professor W. K. Hayman}
\thanks{The first author was supported by Engineering and Physical Sciences Research Council grant EP/R010560/1.\vspace{3pt} }
\begin{abstract}
The set of points where an entire function achieves its maximum modulus is known as the \emph{maximum modulus set}. In 1951, Hayman studied the structure of this set near the origin. Following work of Blumenthal, he showed that, near zero, the maximum modulus set consists of a collection of disjoint analytic curves, and provided an upper bound for the number of these curves. In this paper, we establish the exact number of these curves for all entire functions, except for a ``small'' set whose Taylor series coefficients satisfy a certain simple, algebraic condition. 

Moreover, we give new results concerning the structure of this set near the origin, and make an interesting conjecture regarding the most general case. We prove this conjecture for polynomials of degree less than four.
\end{abstract}
\maketitle

\section{Introduction}
Suppose that $f$ is an entire function, and define the \emph{maximum modulus} by
\[
M(r, f) \defeq \max_{|z| = r} |f(z)|, \qfor r \geq 0.
\]
In the notation of \cite{Sixsmithmax}, the set of points where $f$ achieves its maximum modulus, which we call the \emph{maximum modulus set}, is denoted by $\M(f)$. In other words,
\begin{equation}
\label{Mdef}
\M(f) \defeq \{ z \in \C \colon |f(z)| = M(|z|, f) \}.
\end{equation} 

If $f(z) \defeq cz^n$, for $c \in \C\setminus\{0\}$ and $n \geq 0$, then $\M(f) = \C$. Otherwise $\M(f)$ consists of a union of closed \emph{maximum curves}, which are analytic except at their endpoints; see \cite[Theorem 10]{valironlectures} or \cite{Blumenthal}. Many authors have studied the maximum modulus set; see, for example, \cite{hardy1909, jassimlondon, letidave, letidavenew, tyler,Csordas_1990,Blu_conj}. The maximum modulus set of two cubic polynomials is shown in Figure \ref{fig1}.

It is a simple observation that if $a \ne 0$, $m \in \Z$, and $\tilde{f}(z) \defeq a z^m f(z)$ for entire functions $\tilde{f}$ and $f$, then $\M(\tilde{f}) = \M(f)$. Thus, following Hayman \cite{Hayman}, we will assume that $f$ has the form
\begin{equation}
\label{eq:p}
f(z) \defeq 1 + a z^k + \text{higher order terms}, \qfor a \ne 0, \text{ and } k \in \N.
\end{equation}
Throughout the paper $f$ always has this form, and, in particular, the variables $a$ and $k$ are fixed by this equation.

We are interested in the structure of $\M(f)$ near the origin. Hayman \cite[Theorem I part (iii)]{Hayman} proved the following.
\begin{thm}
If $f$ is an entire function of the form \eqref{eq:p}, then, near the origin, $\M(f)$ consists of at most $k$ analytic curves only meeting at zero. Moreover, for any two of these curves there exists $m\in \Z$ such that the curves make an angle of $2m\pi/k$ with each other.
\end{thm}

In this paper, we strengthen Hayman's result by giving the exact number of such curves for any entire function outside an exceptional set. To give a precise statement of this set, we require the following definitions, the first of which is straightforward.

\begin{definition}
Let $f$ be an entire function  of the form \eqref{eq:p}. We define the \emph{inner degree} of $f$ as the maximal $\mu\defeq \mu_f\in \N$ such that $f(z) = \tilde{f}(z^{\mu})$ for some  entire function $\tilde{f}$.
\end{definition}
Note that in fact $\mu$ is the greatest common divisor of $\{ n>0 \colon f^{(n)}(0)\neq 0\}$, and so it always divides $k$. The second definition is more complicated. Suppose that $f$ is an entire function of the form \eqref{eq:p}, so that we can write
\[
f(z) \defeq 1 + az^k + \sum_{\sigma=k+1}^{\infty} b_\sigma z^{\sigma}.
\]
Let $p_k(z) \defeq 1 + az^k$, and for each $n> k$, define $p_n(z) \defeq 1 + az^k + \sum_{\sigma=k+1}^{n} b_\sigma z^{\sigma}$. It is immediate that there is some least $N \geq k$ such that $\mu_{p_{N}} = \mu_f$. We then say that $p_{N}$ is the \textit{core polynomial} of $f$. Moreover, we stress that $f$ may itself be a polynomial, and it is possible that $p_{N}= f$.
\begin{definition}\label{def:exceptional}
Suppose $f$ is an entire function of the form \eqref{eq:p}, and let $N$ be as defined above. We say that $f$ is \emph{exceptional} if there exist $m \in \{1, \ldots,2k-3\} $, $m' \in \Z$, and $\sigma \in \{k+1, \ldots, N\}$, such that $b_\sigma \ne 0$ and also
\begin{equation}
\label{eq:notmagic}
m\pi = \frac{k}{\sigma}(m'\pi - \arg b_\sigma) + \arg a. 
\end{equation}
\end{definition}
Observe that it is straightforward to determine if an entire function is exceptional, simply by examining the coefficients in its Taylor series. Indeed, we only need to check finitely many such coefficients even when $f$ is transcendental. Note also that no polynomial $p$ with only two terms is exceptional; indeed, it is easy to explicitly check the conclusion of Theorem~\ref{th:2new}, below, in this case.

Our first result establishes the number of curves that form $\M(f)$ near the  origin for any $f$ that is not exceptional.
\begin{theorem} 
\label{th:2new}
Let $f$ be an entire function of the form \eqref{eq:p} that is not exceptional. Then, near the origin, $\M(f)$ consists of exactly $\mu_f$ analytic curves that only meet at zero.
\end{theorem}

\begin{remark}\normalfont
Note that Theorem \ref{th:2new} tells us, in a precise sense, that for ``most'' entire functions, $f$, the set $\M(f)$ has $\mu_f$ components near the origin. 
For, if $f$ is exceptional, then any sufficiently small perturbation of finitely many of its coefficients gives rise to  an entire function that is not exceptional.
\end{remark}

In addition, we are able to provide in Theorem \ref{th:1} more information on the number and asymptotic behaviour of the curves that make up $\M(f)$ near the origin, for any entire function $f$. Set $\Sigma\defeq  \{0,\ldots, k-1\}$. For each $j \in \Sigma$, define the angle
\begin{equation}
\label{eq:angledef}
\omega_j \defeq \frac{2j\pi-\arg a}{k},
\end{equation}
and sectors,
\begin{equation}
\label{eq:sectordef}
S_j(r,\phi)\defeq \{ z \in \C : 0 < |z| \leq r \text{ and } \left|\arg z - \omega_j\right| < \phi \}, \qfor \phi, r>0.
\end{equation}
For a finite set $A$, we use $\# A$ to denote the number of elements of $A$. For an entire function $f$ and a set $T \subset \C$ we set
\[
\M(f\vert_{T}) = \{ z \in T : |f(z)| = \max_{w \in T : |w| = |z|} |f(w)| \}.
\]
\begin{theorem}
\label{th:1}
Suppose that $f$ is an entire function of the form \eqref{eq:p}. Then there exist $R > 0$, a set $J\defeq J_f\subset \Sigma$, and disjoint analytic curves $\{\gamma_j\}_{j\in \Sigma}$ such that 
\begin{equation}
\label{eq:Mindisc}
\M(f) \cap \{ z : 0<|z| \leq R \} = \bigcup_{j \in J} \gamma_j, 
\end{equation}
and with the following properties.
\begin{enumerate}[(a)]
\item\label{item:localmax} There exists $\phi>0$ such that, for each $j\in \Sigma$, $\gamma_j= \M(f\vert_{S_{j}(R,\phi)})$. 
\item\label{th:curves} Each $\gamma_j$ contains exactly one point of each positive modulus less than or equal to~$R$. 
\item Each $\gamma_j$ is tangent at the origin to the ray $\{ z \in \C : \arg z = \omega_j\}$. In particular, $\arg z = \omega_j + O(|z|^{1/2})$ as $z \rightarrow 0$ along $\gamma_j$. \label{th:tangent}

\item\label{th:number} The cardinal $\#J$ is a multiple of the inner degree $\mu$ of $f$. Moreover, if $j,j'\in \Sigma$ and $j' = j + mk/\mu$ with $m\in \N$ so that $0 \leq m<\mu$, then $\gamma_{j'}=e^{2\pi i m/\mu}\gamma_j$, and $j'$ is in $J$ if and only if $j \in J$.
\end{enumerate}
\end{theorem}
We remark that Theorem \ref{th:1} is not completely new; Blumenthal's results (see \cite[II.3]{valironlectures}) imply that near the origin, $\M(f)$ is a finite collection of closed analytic curves. Both the upper bound on the number of curves in \eqref{eq:Mindisc}, and the first part of \ref{th:tangent}, appeared in \cite[Theorem 1]{Hayman}. However, we obtain more explicit estimates and include proofs for completeness.
\begin{remark}\normalfont
Theorem~\ref{th:1}\ref{th:number} implies the following. 
The components of $\M(f)$ near the origin are contained in a disjoint union of families of analytic curves. Each family contains $\mu_f$ such curves, and the curves within each family are obtained from each other by rotations of $2\pi/\mu_f$ radians around the origin. There is at least one of these families, and at most $k/\mu_f$.
\end{remark}

Observe that for an entire function $f$, Theorem \ref{th:1} states in particular that the number of components of $\M(f)$ near the origin is at least its inner degree, that is, $\#J_f \geq \mu_f$. We distinguish the case of strict inequality.
\begin{definition} 
We say that an entire function $f$ of the form \eqref{eq:p} is \emph{magic} if $\#J_f> \mu_f$.
\end{definition}

Theorem~\ref{th:2new} tells us that all magic entire functions are exceptional. The simplest example of a polynomial that is magic seems to be the cubic
\[
p(z) \defeq 1 + z^2 + iz^3, 
\]
see Figure~\ref{fig1} and also Theorem~\ref{th:3}. It is an open question to identify necessary and sufficient conditions for an entire function to be magic. It is also an open question to establish the size of $\#J_f$ in the case where $f$ is magic. We conjecture the following, which, if true, would give a complete answer to the question of the number of disjoint curves in $\M(f)$ near the origin.
\begin{conjecture}
\label{con1}
If $f$ is magic, then $\#J_f = 2\mu_f$.
\end{conjecture}

Although we have not been able to identify all magic entire functions, the following gives a complete result for quadratic and cubic polynomials.
\begin{theorem}
\label{th:3} Suppose that $p$ is a polynomial of the form \eqref{eq:p}. If $p$ is a quadratic, then $p$ is not magic. If $p$ is a cubic, then $p$ is magic if and only if $$p(z) = 1 + az^2 + bz^3, \text{ where } a, b \ne 0 \text{ and } \operatorname{Re}(ba^{-3/2}) = 0.$$
\end{theorem}
\begin{remark}\normalfont
It is straightforward to check that Theorem~\ref{th:3} implies that, for cubic polynomials, $p$ is exceptional exactly when $p$ is magic. It is tempting to conjecture that the same holds more generally.
\end{remark}

The following is an immediate consequence of the proof of Theorem~\ref{th:3}.
\begin{corollary}
Conjecture~\ref{con1} holds for polynomials of degree less than four.
\end{corollary}

\begin{remark} For ease of exposition, we have stated our results for entire functions. However, our arguments only require the existence of a Taylor series locally. Thus, with a suitable definition of the maximum modulus set, our results can be applied to any function analytic in a neighbourhood of the origin.
\end{remark}
We observe finally that, if $p$ is a polynomial, then our results can also be used to study the structure of $\M(p)$ near infinity. This is for the following reason. Suppose that the degree of $p$ is $n$, and let $q$ be the \emph{reciprocal polynomial}, defined by
\[
q(z) \defeq z^n p(1/z).
\]
As observed in \cite[Proposition 3.3]{letidavenew}, we have that $z \in \M(q) \setminus \{0\}$ if and only if $1/z \in \M(p) \setminus \{0\}$. Hence the structure of $\M(p)$ near infinity is completely determined by the structure of $\M(q)$ near the origin.

\subsection*{Acknowledgments} We would like to thank Peter Strulo for programming assistance leading to Figure \ref{fig1}, and Argyrios Christodoulou for helpful feedback. We are very grateful to the referee, whose thoughtful and meticulous report has greatly improved this paper.
%
%
\section{Proof of Theorem~\ref{th:1}}
\label{S.th1}
The goal of this section is to prove Theorem~\ref{th:1}. We use the following, which is easy to check.
\begin{lemma}
\label{lemm:modulus}
If $f(z) \defeq \sum^\infty_{\ell=0} a_\ell z^\ell$ is an entire function, then 
\begin{equation}
\label{eq:modp}
\left| f(r e^{i\theta})\right|^2= \sum_{\ell=0}^\infty \vert a_\ell\vert^2r^{2\ell} +  \sum_{0 \leq j < \ell} 2 \vert a_j\vert \vert a_\ell\vert r^{j+\ell} \cos((j-\ell)\theta+\arg(a_j)-\arg(a_\ell)).
\end{equation}
\end{lemma}

For the rest of the section, let us fix an entire function $f$ as in \eqref{eq:p}, that is,
\begin{equation*}
f(z) \defeq 1 + a z^k + \text{higher order terms}, \qfor a \ne 0, \text{ and } k \in \N.
\end{equation*}
Suppose that $z=r e^{i\theta}$. Then, using Lemma~\ref{lemm:modulus},
\begin{equation}
\label{maineq}
\left| f(r e^{i\theta})\right|^2 = \Lambda(r) +  2|a|r^{k} \cos(k\theta + \arg a) + O(r^{k+1}), \quad \text{ as }\quad r \rightarrow 0,
\end{equation}
where
\[
\Lambda(r) \defeq 1 + |a|^2 r^{2k} + \ldots,
\]
is independent of $\theta$. 

\begin{observation}\label{obs_maineq}
It follows by inspection of \eqref{eq:modp} that all partial derivatives with respect to $\theta$ of the $O(\cdot)$ term in \eqref{maineq} are also $O(r^{k+1})$. 
\end{observation}

Recall from the introduction that for each $j\in \Sigma$, we defined the angle $\omega_j$ in \eqref{eq:angledef} and, for $\phi, r >0$, the sector $S_j(r, \phi)$ in \eqref{eq:sectordef}.

\begin{proof}[Proof of Theorem \ref{th:1}]
Observe that $\M(f)$ is contained in the set of points where $\left| f(r e^{i\theta})\right|$ is locally maximised, that is,
$$\M(f)\subseteq \left\{r e^{i\theta}\in \C \colon \frac{\partial}{\partial \theta} \left| f(r e^{i\theta})\right|=0 \text{ and }\frac{\partial^2}{\partial \theta^2}\left| f(r e^{i\theta})\right| \leq 0 \right\}. $$
Using \eqref{maineq}, and by Observation \ref{obs_maineq}, we have that
\begin{equation}
\label{eq:partial}
\frac{\partial}{\partial \theta} \left| f(r e^{i\theta})\right|^2 = -2|a| k r^{k} \sin(k\theta + \arg a) + O(r^{k+1}), \text{ as } r \rightarrow 0,
\end{equation}
and
\begin{equation}\label{eq_secondpartial}
\frac{\partial^2}{\partial \theta^2} \left| f(r e^{i\theta})\right|^2 = -2|a|k^2 r^{k} \cos(k\theta + \arg a) + O(r^{k+1}), \text{ as } r \rightarrow 0.
\end{equation}

Fix $r_1,  \phi >0$ sufficiently small, with the property that for all $0<r\leq r_1$ and $r e^{i\theta}\in  \bigcup^{k-1}_{j=0} S_j(r_1,\phi)$, the second derivative in \eqref{eq_secondpartial} is not positive. Reducing $r_1$ and $\phi$ if necessary, we can deduce that for each $0< r\leq r_1$ and $j\in \Sigma$, there is exactly one point $re^{i\theta} \in S_j(r_1,\phi)$ at which the derivative in \eqref{eq:partial} is zero; the fact there is at least one such point follows from \eqref{eq:partial}, and the fact there is at most one follows from \eqref{eq_secondpartial}. Moreover, $\cos(k\theta + \arg a)$ takes the value 1 inside each sector, and is bounded above by a quantity less than 1 which depends only on $\phi$ outside the union of sectors. Now it follows from \eqref{maineq} that
\begin{equation}\label{eq_rev}
\M(f) \cap \{ z : 0<|z| \leq r_1 \}\subseteq \bigcup^{k-1}_{j=0} S_j(r_1,\phi).
\end{equation}

Next, for each $j\in \Sigma$ and $0 < r \leq r_1$, let 
$$\gamma^r_j\defeq \M(f\vert_{S_{j}(r,\phi)}).$$

Note that $\gamma^{r_1}_j$ is the solution set in $S_j(r_1, \phi)$ to \eqref{eq:partial} being zero. Using a change of variables or the implicit function theorem, see \cite[Lemma~4]{Hayman} or \cite[II.3]{valironlectures}, one can see that $\gamma^{r_1}_j$ is an analytic curve. It is easy to see that $\gamma^{r_1}_j$ contains exactly one point of each modulus.

Thus, we have shown that there exists $r_1>0$ and a collection $\{\gamma^{r_1}_j\}_{j\in \Sigma}$ of disjoint analytic curves such that $\gamma^{r_1}_j=\M(f\vert_{S_{j}(r_1,\phi)})$. By this and \eqref{eq_rev}, we have
$$\M(f) \cap \{ z : 0<|z| \leq r_1 \} \subseteq \bigcup_{j\in \Sigma}\gamma^{r_1}_j.$$
By results of Blumenthal \cite{Blumenthal}, see \cite[Section 3]{letidavenew}, it follows that there exists $R<r_1$ such that $\M(f)\cap \{ z : 0<|z| \leq R \}= \bigcup_{j\in J}\gamma^{R}_j$ for some subset $J\subseteq \Sigma$.  We deduce \ref{item:localmax} and \ref{th:curves}. 

Next we prove \ref{th:tangent}. First, note that by \eqref{maineq}, for each $j\in J$, the curve $\gamma_j\defeq \gamma^{R}_j$ is asymptotic to the set of points where the term $\cos(k\theta + \arg a)$ is maximised. It follows that $\gamma_j$ is tangent at the origin to the ray $L_j$ of argument $\omega_j$. It remains to estimate at what rate points of $\gamma_j$ tend to $L_j$ as we move towards the origin.

For each $0 < r \leq R$, denote the argument of the point of $\gamma_j$ of modulus $r$ by $\omega_j + \theta_r$. Fix $j$, and let $z = re^{i (\omega_j + \theta_r)} \in \gamma_j$. Then, by \eqref{maineq}, as $r \rightarrow 0$, we have
\begin{align*}
|f(re^{i\omega_j})|^2 \!-\! |f(z)|^2 &= 2|a|r^k \cos(\arg a \!+\!k\omega_j)\!-\!2|a| r^k \cos(\arg a \!+\! k\omega_j \!+\! k\theta_r)\! + \! O(r^{k+1}) \\
                                 &= 2|a|r^k(1 - \cos(k\theta_r) + O(r)).
\end{align*}
Since $|f(re^{i\omega_j})|^2 - |f(z)|^2$ is not positive, neither is $(1 - \cos(k\theta_r) + O(r))$. Since $1-\cos(k\theta_r) \geq (k\theta_r)^2/3$ when $\theta_r$ is small, it follows that $\theta_r = O(r^{1/2})$ as $r \rightarrow 0$. We deduce \ref{th:tangent}.

From the definition of $\mu$, it follows that $f(z) = f(z e^{2\pi i n/\mu})$, for every $z$ and every integer $n$. With $j, j'$, and $m$ as in the statement of \ref{th:number}, it follows that $z$ is in the sector $S_j(R, \phi)$ if and only if $ze^{2\pi i m/\mu}$ is in $S_{j'}(R, \phi)$. Combining these two facts, we obtain the desired relationship between $\gamma_j$ and $\gamma_{j'}$ and also conclude that $j \in J$ if and only if $j'  \in J$. Finally, note that considering the relation $j' \equiv j \mod k/\mu$ we can divide $\Sigma$ into $k/\mu$ equivalence classes of $\mu$ elements each, and we have shown that $J$ consists of a union of some of these equivalence classes. This proves \ref{th:number}, which completes the proof of the theorem.
\end{proof}
%
%
\section{Auxiliary results}
\label{S:auxilliary}
To prove Theorem~\ref{th:2new}, we need to prove in Section \ref{S:th2} a key result on the maximum modulus set of certain polynomials. In this section we give some auxiliary results on the maximum modulus of any polynomial $p$ of the form \eqref{eq:p}, which we state separately since they do not require any further assumptions on $p$. In particular, these results may be useful for future applications. Then, we state and prove the key result, namely Theorem~\ref{th:2}, in Section \ref{S:th2}.

Throughout this section and Section \ref{S:th2}, we fix a polynomial $p$ of the form \eqref{eq:p}. Note that if $p(z) = 1 + az^k$, then Theorem~\ref{th:2new} follows trivially. Hence we can assume that $p$ has at least three terms. Let $\hp$ be the polynomial of degree less than $p$ such that
\begin{equation}
\label{eq:hpdef}
p(z) = \hp(z) + b z^n, 
\end{equation}
for some $b \ne 0$ and $n \in \N$ the degree of $p$. Note that, in particular, $\hp$ is a polynomial of the form \eqref{eq:p}, whose non-constant term of least degree is the same as that of $p$, that is, $az^k$. The polynomials $p$ and $\hp$ will remain fixed from now on.

We next introduce some notation that will be used extensively in both this section and the next. By Theorem~\ref{th:1}\ref{th:number} we know that $J_p$ consists of one or more disjoint sets, each of which contains all the elements of $\Sigma$ that are congruent modulo $k/\id$. If $j \in J_p$, then we use $[ j ]_p$ to denote this set, that is,
\[
[ j ]_p \defeq \{ j' \in \Sigma : j' \equiv j \mod k/\mu_p \}.
\]
In addition, let $R>0$, and let $\{\gamma_j\}_{j\in \Sigma}$ be the collection of curves provided by Theorem \ref{th:1}, so that \eqref{eq:Mindisc} holds. Then, for $0 < r \leq R$ and $j \in \Sigma$, we let $z_j(r)$ denote the unique point on $\gamma_j$ of modulus $r$. Moreover, reducing $R$ if necessary, if $\{\hat{\gamma}_j\}_{j\in \Sigma}$ is the corresponding set of curves provided by Theorem \ref{th:1} applied to $\hp$, then we let $\hz_j(r)$ denote the unique point on $\hat{\gamma}_j$ of modulus $r$. This completes the definition of the notation.

We next give an observation which allows us to estimate the square of the modulus of $p$ in terms of that of $\hp$ at each point in the plane. This is an immediate consequence of Lemma \ref{lemm:modulus}.
\begin{observation}
\label{obs:approx}
We have, as $r \rightarrow 0$, that
\begin{equation}\label{eq:difference}
|p(re^{i\theta})|^2 - |\hp(re^{i\theta})|^2= 2|b|r^n\cos(n\theta+ \arg b)+ O(r^{n+k}),
\end{equation}
 and all partial derivatives of the $O(\cdot)$ term with respect to $\theta$ are also $O(r^{n+k})$. 
\end{observation}
The first lemma in this section is key to the proof of Theorem \ref{th:2}. Roughly speaking, it says that, close to the origin, $|\hp|^2$ at a point $\hz_j(r)$ (which is where $\hp$ takes its maximum modulus) is very close to $|\hp|^2$ at a point $z_j(r)$ (which is where $p$ takes its maximum modulus). 
\begin{lemma}
\label{lemm:phatisnice}
Suppose $j \in \Sigma$. Then
\[
|\hp(\hz_{j}(r))|^2 - |\hp(z_j(r))|^2 = O(r^{n+1/2}), \text{ as } r \rightarrow 0.
\]
\end{lemma}
\begin{proof}
To prove this result, it helps to simplify notation. Suppose $j \in \Sigma$ is fixed. Then, define real analytic functions $f,g,h\colon \R\times \R\to \R$ as
\[
f(r,\theta) \defeq |p(r e^{i(\omega_j + \theta)})|^2, \qquad g(r,\theta) \defeq |\hp(r e^{i(\omega_j + \theta)})|^2,
\]
and finally,
\[
h(r,\theta) \defeq 2|b|r^n \cos(\arg b + n (\omega_j + \theta)).
\]

We use a dash to denote differentiation with respect to $\theta$. We need to estimate the partial derivatives of $f, g$ and $h$. It follows from (\ref{maineq}) that 
\[
g'(r,\theta) = -2|a|kr^k\sin(k\theta + k\omega_j + \arg a)+O(r^{k+1}), 
\]
and 
\begin{equation}
\label{eq:gdoubledash}
g''(r,\theta)= -2|a|k^2r^k\cos(k\theta + k\omega_j + \arg a)+O(r^{k+1}).
\end{equation}
Note that all derivatives of $f$ and $g$ with respect to $\theta$ are $O(r^k)$ as $r \rightarrow 0$, because the first term dominates. We also have that 
\begin{equation}
\label{eq:hdash}
h'(r,\theta)= -2|b|nr^n\sin(n\theta + n\omega_j + \arg b),
\end{equation}
and all derivatives of $h$ with respect to $\theta$ are $O(r^n)$ as $r \to 0$. Finally, it follows from the definitions, along with Observation \ref{obs:approx}, that
\begin{equation}
\label{eq:fgh}
f(r,\theta)= g(r,\theta)+h(r,\theta)+O(r^{n+k}),
\end{equation}
\begin{equation}
\label{eq:derivatives}
f'(r,\theta)= g'(r,\theta)+h'(r,\theta)+O(r^{n+k}),
\end{equation}
and all the higher order derivatives of the $O(\cdot)$ term are also $O(r^{n+k}).$

Recall that for $r$ sufficiently small, $z_j(r)$ and $\hz_j(r)$ are the respective points in the curves indexed by $j$ where $p$ and $\hp$ attain the maximum modulus. Let us write 
$$z_j(r) = r e^{i(\omega_j +\theta_r)} \quad\text{and}\quad \hz_j(r) = r e^{i(\omega_j + \hat{\theta}_r)},$$ 
where the angles $\theta_r$ and $\hat{\theta}_r$ are both functions of $r$. In particular, it follows from the definitions that
\[
f'(r,\theta_r) = 0 \quad\text{ and }\quad g'(r,\hat{\theta}_r) = 0.
\]

With this notation, our goal in this lemma is to estimate, for small values of $r>0$, the quantity $g(r,\hat{\theta}_r) - g(r,\theta_r)$. Recall that, by Theorem~\ref{th:1}\ref{th:tangent}, $\theta_r$ and $\hat{\theta}_r$ are both $O(r^{1/2})$ as $r\to 0$. Since $f$ is real analytic, since $g'(r,\hat{\theta}_r) = 0$, and as $\theta_r - \hat{\theta}_r = O(r^{1/2})$, it follows by \eqref{eq:derivatives} and our bounds on the derivatives that
\begin{align*}
f'(r,\theta_r) &= f'(r,\hat{\theta}_r)+ (\theta_r -\hat{\theta}_r) f''(r,\hat{\theta}_r)+ O((\theta_r - \hat{\theta}_r)^2 \cdot r^{k}) \\
               &= h'(r,\hat{\theta}_r) + (\theta_r -\hat{\theta}_r)(g''(r,\hat{\theta}_r) + h''(r,\hat{\theta}_r)) + O(r^{n+k}) + O((\theta_r - \hat{\theta}_r)^2 \cdot r^{k}).
\end{align*}
Since $f'(r,\theta_r) = 0$, we can deduce that
\[
(\theta_r - \hat{\theta}_r)(g''( r,\hat{\theta}_r) + h''(r,\hat{\theta}_r)) = -h'(r,\hat{\theta}_r)  + O(r^{n+k})+ O((\theta_r - \hat{\theta}_r)^2 \cdot r^{k}).
\]

Moreover, note that by \eqref{eq:gdoubledash} and the bounds on $h''$, we have that
\[
g''( r,\hat{\theta}_r) + h''(r,\hat{\theta}_r) = -2|a|k^2r^k\cos(k\hat{\theta}_r + k\omega_j + \arg a)+O(r^{k+1}),
\]
and $\cos(k\hat{\theta}_r + k\omega_j + \arg a)$ can be taken to be greater than $1/2$.
It then follows by \eqref{eq:hdash} that 
\begin{equation}
\label{eq:thetar}
|\theta_r - \hat{\theta}_r| 
= O(r^{n-k}) +  O((\theta_r - \hat{\theta}_r)^2). 
\end{equation}
Since $\theta_r - \hat{\theta}_r = O(r^{1/2})=o(1)$, and since $O((\theta_r - \hat{\theta}_r)^2)=o(\theta_r - \hat{\theta}_r)$
it follows that
\begin{equation}\label{eq:claim}
|\theta_r - \hat{\theta}_r| = O(r^{n-k}), \text{ as } r \rightarrow 0.
\end{equation}

We can now deduce from \eqref{eq:claim}, together with  {\eqref{eq:fgh}}, \eqref{eq:derivatives}, the real analyticity of $f$ and $h$, and our earlier estimates for the size of $f''$ and $h''$, that
\begin{align*}
g(r,\hat{\theta}_r) - g(r,\theta_r) &= (f(r,\hat{\theta}_r) - f(r,\theta_r))- (h(r,\hat{\theta}_r) - h(r,\theta_r)) + O(r^{n+k}) \\
                                  &= (\hat{\theta_r} - \theta_r)(f'(r,\hat{\theta}_r) -h'(r,\hat{\theta}_r)) + O((\hat{\theta_r} - \theta_r)^2 \cdot r^{k}) + O(r^{n+k}) \\
												          &= (\hat{\theta_r} - \theta_r) g'(r,\hat{\theta}_r) + O(r^{2n}) + O(r^{2n-k}) + O(r^{n+k}) \\
												          &= O(r^{n+1/2}),
\end{align*}
as required. Note that in the last step we have used that $g'(r,\hat{\theta}_r) = 0$ and also that $2n - k> n + 1/2$, since $n \geq k+1$. 
\end{proof}
Now, for each $j \in \Sigma$, we let $t_j \defeq 2|b|\cos(n\omega_j+\arg b)$, where we recall that $\omega_j \defeq (2j\pi-\arg a)/k$. Our next lemma allows us to compare the magnitude of $p$ on different $\gamma_j$. 
\begin{lemma}
\label{lem_newtj} 
Let $j,j'\in J_{\hp}$. Then
\begin{equation}
\label{eq:diff}
|p(z_j(r))|^2 - |p(z_{j'}(r))|^2 = (t_j - t_{j'}) r^n + O\left(r^{n+1/2}\right) \text{ as } r \rightarrow 0.
\end{equation}
\end{lemma}
\begin{proof}
Let us write 
$$z_j(r) = r e^{i(\omega_j +\theta_r)} \quad\text{and}\quad z_{j'}(r) = r e^{i(\omega_{j'} + \theta'_r)},$$ 
where the angles $\theta_r$ and $\theta'_r$ are functions of $r$. Then, by Theorem~\ref{th:1}\ref{th:tangent}, $\theta_r$ and $\theta'_r$ are both $O(r^{1/2})$ as $r\to 0$. By this, and by Observation \ref{obs:approx}, as $r \rightarrow 0$ we have that
\begin{align*}
|p(z_j(r))|^2 &= |\hp(z_j(r))|^2 + 2|b|r^n\cos(n\omega_j + n\theta_r+ \arg b)+ O(r^{n+k}), \\
              &= |\hp(z_j(r))|^2 + 2|b|r^n(\cos(n\omega_j + \arg b)\cos n\theta_r - \sin(n\omega_j + \arg b)\sin n\theta_r) + O(r^{n+k}), \\
              &= |\hp(z_j(r))|^2 + t_j r^n + O\left(r^{n+1/2}\right),
\end{align*}
where in the last line we have used that $\cos n\theta_r = 1 + O(r)$ and $\sin n\theta_r = O(r^{1/2})$. Also, arguing similarly, we have
\[
|p(z_{j'}(r))|^2=|\hp(z_{j'}(r))|^2 + t_{j'} r^n + O\left(r^{n+1/2}\right), \text{ as } r \rightarrow 0. 
\]

Now, $|\hp(\hz_{j'}(r))| = |\hp(\hz_{j}(r))|$ since  $j,j'\in J_{\hp}$. Hence, by Lemma~\ref{lemm:phatisnice}, we have
\begin{align*}
|p(z_j(r))|^2 - |p(z_{j'}(r))|^2 &= (|\hp(z_j(r))|^2-|\hp(\hz_{j}(r))|^2) - (|\hp(z_{j'}(r))|^2-|\hp(\hz_{j'}(r))|^2) \\
                             &+ (t_j - t_{j'}) r^n + O\left(r^{n+1/2}\right) \\
                             &= (t_j - t_{j'}) r^n + O\left(r^{n+1/2}\right).\qedhere
\end{align*}
\end{proof}

It follows from Lemma~\ref{lem_newtj} that the magnitudes of the quantities $t_j$, for $j \in \Sigma$, are important for determining the size of $|p(z)|$. It proves useful to know exactly when two of these terms can be equal. This is the content of the following lemma.
\begin{lemma}
\label{lemm:equals}
Suppose $j, j' \in \Sigma$. Then $t_j = t_{j'}$ if and only if there is an integer $m$ such that one of the following holds. Either
\begin{equation}
\label{poss1}
j' - j = m\frac{k}{n},
\end{equation}
or
\begin{equation}
\label{poss2}
(j' + j)\pi = \frac{k}{n}(m\pi - \arg b) + \arg a. 
\end{equation}
\end{lemma}
\begin{proof}
This is a straightforward consequence of the fact that $\cos z_1 = \cos z_2$ if and only if there is an integer $m$ such that either $z_1 = z_2 + 2m\pi$, or $z_1 = 2m\pi - z_2$. The details are omitted.
\end{proof}
The next lemma gives a simple relationship between the condition \eqref{poss1} and the sets $[j]_p$.
\begin{lemma}
\label{lemm:reduce}
Suppose that $j, j' \in J_{\hp}$ with $j' \in [j]_{\hp}$. Then \eqref{poss1} holds if and only if $j' \in [j]_p$.
\end{lemma}
\begin{proof}
First, we note that since $p(z)= 1+ az^k+\dots+bz^n= \hat{p}(z) +bz^n,$ and $\id \leq \hid$, there are natural numbers $A_0, A_1, A_2$ such that $\hid = A_0 \id$, $k = A_1 \hid$ and $n = A_2 \id$. Moreover, $A_0$ and $A_2$ are coprime, since if they shared a factor $A_3 > 1$, then we could replace $\id$ with $A_3 \id$.

Since $j' \in [j]_{\hp}$, it follows by the definition of $[j]_{\hp}$ that there is an integer $B_0$ such that
\[
j' - j = \frac{B_0k}{\hid}.
\]

Suppose first that \eqref{poss1} holds. We can deduce that $m A_0 = B_0 A_2$. Since $m$ is an integer and $A_0$ and $A_2$ are coprime, $m = B_1 A_2$, where $B_1$ is an integer. Hence
\[
j' - j = \frac{B_1 A_2 k}{A_2 \id} = B_1 \frac{k}{\id},
\]
as required.

In the other direction, suppose that $j' \in [j]_p$. Then there is an integer $B_1$ such that 
\[
j' - j = B_1 \frac{k}{\id}.
\]
It follows that
\[
j' - j = B_1 A_2 \frac{k}{n},
\]
and so \eqref{poss1} holds.
\end{proof}
Our last general lemma allows us to compare $J_q$ and $J_{\tq}$ for two related entire functions $q, \tq$ of the form \eqref{eq:p}, where $q$ is a polynomial and $\tq$ may be a polynomial or may be transcendental. Note that $J_q$ and $J_{\tq}$ are the subsets of $\Sigma$ provided by Theorem \ref{th:1}, and these are both well defined sufficiently close to the origin.

\begin{lemma}
\label{lemm:general}
Suppose that $q$ is a polynomial of the form \eqref{eq:p}, of degree $n$. Let $\{\gamma_{\ell}\}_{{\ell}\in \Sigma}$ be the set of curves provided by Theorem \ref{th:1} applied to $q$, and for all sufficiently small values of $r$, let $z_{\ell}(r)$ denote the point on $\gamma_{\ell}$ of modulus $r$. Suppose also that there exist $c, R > 0$ such that
\begin{equation}
\label{eq:condition}
|q(z_{j}(r))|^2 \geq |q(z_{j'}(r))|^2 + c r^{n}, \qfor j \in J_q, j' \in \Sigma \setminus J_q, \text{ and } 0 < r \leq R.
\end{equation}
If $\tq$ is an entire function such that, for some $\tilde{n} > n$ and $b \ne 0$, its power series is
\begin{equation}
\label{eq:gandq}
\tq(z) = q(z) + bz^{\tilde{n}} + \ldots,
\end{equation}
then $J_{\tq} \subset J_q$.
\end{lemma}
\begin{proof}
Suppose, by way of contradiction, that there exists $j' \in J_{\tq} \setminus J_q$. Choose $j \in J_{q}$, and let $r > 0$ be small. Let $\{\tilde{\gamma}\}_{\ell\in \Sigma}$ be the set of curves provided by Theorem \ref{th:1} applied to $\tq$. For all sufficiently small values of $r$, let $\tilde{z}_{\ell}(r)$ denote the point on $\tilde{\gamma}_{\ell}$ of modulus $r$. 

By Theorem \ref{th:1}\ref{item:localmax}, since both $\gamma_{j'}$ and $\tilde{\gamma}_{j'}$ are contained in the same sector $S_{j'}(r,\phi)$ for some $r, \phi>0$, and $\gamma_{j'}=\M(q\vert_{S_{j'}(r,\phi)})$, we have that 
\begin{equation}
\label{eq:insector}
|q(z_{j'}(r))|^2 \geq |q(\tilde{z}_{j'}(r))|^2, \qfor r > 0 \text{ small}.
\end{equation}

By \eqref{eq:gandq}, \eqref{eq:condition}, and \eqref{eq:insector}, there are positive constants $c, K$ such that, for small values of $r > 0$, we have
\begin{align*}
|\tq(\tilde{z}_{j'}(r))|^2 &\leq |q(\tilde{z}_{j'}(r))|^2   +  K^2r^{2\tilde{n}}+ 2Kr^{\tilde{n}}|q(\tilde{z}_{j'}(r))|, \\
                          &\leq |q(z_{j'}(r))|^2             +  K^2r^{2\tilde{n}}+ 2Kr^{\tilde{n}}|q(\tilde{z}_{j'}(r))|, \\
                          &\leq |q(z_{j}(r))|^2 - cr^n   +  K^2r^{2\tilde{n}}+ 2Kr^{\tilde{n}}|q(\tilde{z}_{j'}(r))|, \\
                          &\leq |\tq(z_{j}(r))|^2 - cr^n + 2K^2r^{2\tilde{n}}+ 2Kr^{\tilde{n}}(|q(\tilde{z}_{j'}(r))| + |\tq(z_{j}(r))|).
\end{align*}
For sufficiently small values of $r$ this is a contradiction, since $n < \tilde{n}$ and for such values
\[
|\tq(\tilde{z}_{j'}(r))|^2 \geq |\tq(z_{j}(r))|^2.\qedhere
\]
\end{proof}
%
%
%
\section{Minimal polynomials} \label{S:th2}
In this section we introduce the notion of minimal polynomials, which is closely related to the definition of exceptional ones.
\begin{definition}\label{def_minimal}
Let $p(z) \defeq 1 + az^k + \sum_{\sigma=k+1}^{N} b_\sigma z^{\sigma}$ be a polynomial. We say that $p$ is \emph{minimal} if for all $m \in \{1, \ldots, 2k-3\} $, $m' \in \Z$, and $\sigma \in \{k+1, \ldots, N\}$ such that $b_\sigma \ne 0$,
\begin{equation}\label{eq_min}
m\pi \neq\frac{k}{\sigma}(m'\pi - \arg b_\sigma) + \arg a. 
\end{equation}
\end{definition}
Note that a two-term polynomial is minimal by default, and that, indeed, a polynomial is minimal if and only if it is not exceptional.

The following simple lemma is critical to our arguments in this section, and indeed underlies the definition of ``exceptional''.
\begin{lemma} \label{obs_minimal} Let $p(z) \defeq 1 + az^k + \sum_{\sigma=k+1}^{N} b_\sigma z^{\sigma}$ be a polynomial.
\begin{enumerate}[(a)]
\item For each $k<n\leq N$, define $p_n(z) \defeq 1 + az^k + \sum_{\sigma=k+1}^{n} b_\sigma z^{\sigma}.$ If $p$ is minimal, so is $p_n$ for all $k<n\leq N$, and $p$ is not exceptional.
\item If $f$ is entire and not exceptional, then its core polynomial $p$ is minimal.
\end{enumerate}
\end{lemma}
\begin{proof}
To prove (a) note first that the fact that, for $k<n\leq N$, $p_n$ is minimal when $p$ is follows from the definition of minimal, where in \eqref{eq_min} the variable $\sigma$ might take fewer values for $p_n$ than for $p$. Since \eqref{eq_min} holds for a minimal polynomial $p$, \eqref{eq:notmagic} cannot hold for $p$, and so $p$ cannot be exceptional.\\
For (b) suppose that $f$ is an entire map which is not exceptional. Then, by Definition \ref{def:exceptional}, its core polynomial $p$ is not exceptional, and \eqref{eq_min} must hold for $p$. Hence, $p$ is minimal.
\end{proof}
The following result on minimal polynomials is key to the proof of Theorem~\ref{th:2new}.
\begin{theorem}
\label{th:2}
If $p$ is a minimal polynomial,
\begin{enumerate}[(i)]
\item\label{th:bigger} there exist $c > 0$ and $R$ such that \eqref{eq:condition} holds with $p$ in place of $q$.
\item\label{th:symmetry} There exists $j \in \Sigma$ such that $J_p = [j]_p$. 
\end{enumerate}
\end{theorem}

\begin{proof}[Proof of Theorem~\ref{th:2}]
We prove the result by induction on the number of non-zero terms in $p$. When $p$ contains $2$ non-zero terms, then we have that
\[
p(z) = 1 + az^k.
\]
Clearly, the maximum modulus of $p$ is achieved exactly when $a z^k$ is real and positive, in other words, when $\arg a + k \arg z$ is a multiple of $2\pi$. Then $J_p = \Sigma$, and so Theorem~\ref{th:2}\ref{th:bigger} holds trivially. Note that Theorem~\ref{th:2}\ref{th:symmetry} is also straightforward.

Now suppose the theorem has been proved for up to $\ell$ non-zero terms, and $p$ has $\ell+1$ non-zero terms. Let $\hp$ be the polynomial defined in \eqref{eq:hpdef}. Note that $\hp$ has fewer terms than $p$, and since $p$ is minimal, by Lemma \ref{obs_minimal}, so is $\hp$. Hence, we can assume that both the inductive conclusions apply to $\hp$. 

Observe that, by the inductive hypothesis, and by the definitions of $p$ and $\hp$, the conditions of Lemma~\ref{lemm:general} are satisfied, with $\hp$ in place of $q$ and $p$ in place of $\tq$. Hence $J_p \subset J_{\hp}$.

\begin{claim} Set $\tk \defeq \max \{ t_j : j \in J_{\hp} \}$. Then
\begin{equation} \label{eq_Jp}
J_p = \{ j \in J_{\hp} : t_j = \tk \}.
\end{equation}
\end{claim}
\begin{subproof}
Note first, by Lemma \ref{lem_newtj}, that if $j, j' \in J_{\hp}$, then \eqref{eq:diff} holds. This implies that if $t_{j'} < t_j$, then $j' \notin J_p$. Hence $J_p \subset \{ j \in J_{\hp} : t_j = \tk \}$.
 
For the reverse inclusion, choose $j' \in J_p$, so that $t_{j'} = \tk$. Suppose that $j \in J_{\hp}$ and that $t_j = \tk$. We need to show that $j \in J_p$. Since $t_j = t_{j'}$, it follows from Lemma~\ref{lemm:equals} that either \eqref{poss1} or \eqref{poss2} hold. Since, $p$ is minimal, \eqref{poss2} cannot hold. By Theorem~\ref{th:2}\ref{th:symmetry} applied to $\hp$, we know that $J_{\hp} = [j]_{\hp}$, and so $j' \in J_p \subset J_{\hp} = [j]_{\hp}$. Since \eqref{poss1} holds, we may apply Lemma~\ref{lemm:reduce} to conclude that $j \in [j']_p$ and the claim follows from Theorem \ref{th:1}\ref{th:number}.
\end{subproof}

We now show that Theorem~\ref{th:2}\ref{th:bigger} holds for $p$. To see this, let $j \in J_p$ and $j' \notin J_p$. If $j'\in J_{\hat{p}}$, then $t_{j'} < t_j$ by \eqref{eq_Jp} and Theorem~\ref{th:2}\ref{th:bigger} follows from \eqref{eq:diff}. If $j' \notin J_{\hat{p}}$, then, since $j\in J_{\hat{p}}$, and $\hat{p}$ has fewer terms than $p$, by our inductive assumption \eqref{eq:condition} holds with $q$ replaced by $\hat{p}$ and $n$ replaced by $m$, the degree of $\hat{p}$. Since $m < n$, this latter fact combined with Observation \ref{obs:approx} yields \eqref{eq:condition} with $q$ replaced by $p.$

It remains to show that Theorem~\ref{th:2}\ref{th:symmetry} holds for $p$. Fix $j, j' \in J_p$. We are required to show that $j' \in [j]_p$. We have that $j'\in J_p \subset J_{\hat{p}}=[j]_{\hat{p}}$ by our inductive assumption. Since $p$ is minimal we can deduce that \eqref{poss1} holds. Now by \eqref{eq_Jp} we can  apply Lemma~\ref{lemm:reduce} to conclude $j' \in [j]_p$.
\end{proof}

%
%
\section{Proof of Theorem~\ref{th:2new}}
\label{S:trans}
It remains to use Theorem \ref{th:2} to complete the proof of Theorem~\ref{th:2new}. This is, in fact, quite straightforward. Suppose that $f$ is entire and not exceptional, and is of the form \eqref{eq:p}. Let $p$ be its core polynomial, that by Lemma \ref{obs_minimal} is minimal and has the same inner degree as $f$. Then, by Theorem~\ref{th:2}\ref{th:symmetry}, there exists $j \in \Sigma$ such that $J_p =[j]_p$, and, in particular, $\#J_p=\mu_p.$

Observe that, by Theorem~\ref{th:2}\ref{th:bigger} applied to $p$, and by the definitions of $f$ and $p$, the conditions of Lemma~\ref{lemm:general} are satisfied, with $p$ in place of $q$ and $f$ in place of $\tq$. Hence $J_f \subset J_{p}$. The result now follows since, by Theorem~\ref{th:1}, $J_f$ contains at least $\mu_f$ elements, and $\mu_f = \mu_p$ by definition of $p.$
%
%
\section{Proof of Theorem~\ref{th:3}}
\label{S:th3}

\begin{figure}[htb]
    \centering
    \begin{subfigure}[b]{\textwidth}
        \centering
        \includegraphics[width=0.45\linewidth]{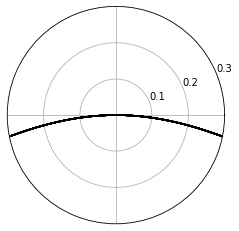}%
        \hfill
        \includegraphics[width=0.45\linewidth]{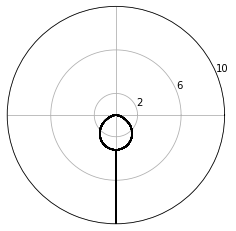}
        \caption{$p(z) \defeq 1+z^2+i z^3$}
    \end{subfigure}
    \vskip\baselineskip
    \begin{subfigure}[b]{\textwidth}
        \centering
        \includegraphics[width=0.45\linewidth]{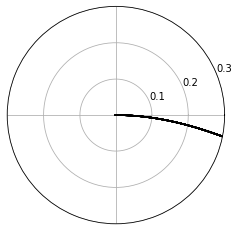}%
        \hfill
        \includegraphics[width=0.45\linewidth]{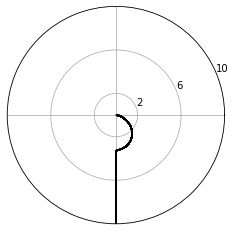}
        \caption{$\tilde{p}(z) \defeq 1+z^2+(i+0.001)z^3$}
    \end{subfigure}
    \caption{\label{fig1} Computer generated graphics of $\M(p)$ and $\M(\tilde{p})$ near the origin. By Theorem \ref{th:3}, the polynomial $p$ is magic and so $\M(p)$ has two components near the origin. The polynomial $\tilde{p}$, which is only a small perturbation of $p$, is not magic, and so $\M(\tilde{p})$ has only one component near the origin.}
\end{figure}

In this section we prove Theorem \ref{th:3}. Suppose that $p$ is a polynomial of the form \eqref{eq:p}, and let $n$ be its degree. Observe that if $k=n$, then its inner degree $\id=k$, and so by Theorem~\ref{th:1}\ref{th:number}, $\#J_p=\id$ and $p$ is not magic. If $k=1$, then by Theorem \ref{th:1}, $\#J_p=1$ and $p$ is not magic. In particular, these are the only two possible cases for quadratic polynomials, and so they cannot be magic. Likewise, if $p$ is cubic, then it can be magic only if $k=2$.

Suppose that $p$ is a cubic polynomial, and that $k=2$; that is, $p(z)=1+az^2+bz^3$, where $a, b \ne 0$. Let $\beta\in \C$ be such that $a\beta^2 = 1$, and set $b' \defeq b \beta^{3} = ba^{-3/2}$. It is easier to consider the polynomial
\begin{equation}\label{eq_pw}
q(z) \defeq p(z\beta) = 1 + z^2 + b'z^3.
\end{equation}

Let $z = re^{i\theta}$ so that $-\overline{z} = re^{i(\pi - \theta)}$, and set $\phi \defeq \arg b'$. By an application of Lemma~\ref{lemm:modulus}, together with standard trigonometric formulae, we can calculate that
\begin{align*}
|q(z)|^2 - |q(-\overline{z})|^2 &= 2r^3|b'|\left(\cos(3\theta\!+\!\phi)\!+\! r^2\cos(\theta \!+\! \phi)\!+\! \cos(3\theta\!-\! \phi)\!+\! r^2\cos(\theta \!-\! \phi)\right) \\
&= 4r^3|b'| \ \cos\phi \ (\cos 3\theta + r^2 \cos \theta).
\end{align*}

Suppose that $|\theta|$ is small, which is the case if $z$ is near the positive real axis. If the real part of $b'$ is positive, then $\cos\phi$ is positive, and we can deduce that $|q(z)|^2 > |q(-\overline{z})|^2$. By Theorem \ref{th:1}\ref{th:tangent}, near the origin, $\M(q)$ lies near the real axis. It follows that $\M(q)$ has only one component near the origin, which is asymptotic to the positive real axis, and $q$ is not magic. 

If the real part of $b'$ is negative, then $\cos\phi$ is negative, and we can deduce that $|q(z)|^2 < |q(-\overline{z})|^2$. As above, it follows that $\M(q)$ has only one component near the origin, which is asymptotic to the negative real axis, and again $q$ is not magic. 

Finally, if $b'$ is imaginary, then $\cos\phi = 0$ and $|q(z)| = |q(-\overline{z})|$; in fact we have $q(z) = \overline{q(-\overline{z})}$. It follows that $\M(q)$ has exactly two components near the origin, which are obtained from each other by reflection in the imaginary axis. In particular, $q$ is magic. 

Since $z \in \M(q)$ if and only if $z\beta \in \M(p)$, we can deduce that $p$ is magic if and only if $b'$ is imaginary, as required.

\bibliographystyle{alpha}
\bibliography{HaymanReferences}
\end{document}